\documentclass{birkjour}

\newtheorem{thm}{Theorem}[section]

\newtheorem{prop}[thm]{Proposition}

\newcommand{\C}{{\mathbb C}}

\newcommand{\R}{{\mathbb R}}
\newcommand{\T}{{\mathbb T}}
\newcommand{\Z}{{\mathbb Z}}
\newcommand{\N}{{\mathbb N}}

\newcommand{\La}{\Lambda}

\newcommand{\f}{\frac}
\newcommand{\ov}{\overline}
\newcommand{\al}{\alpha}
\newcommand{\be}{\beta}

\newcommand{\ga}{\gamma}

\newcommand{\ze}{\zeta}
\renewcommand{\th}{\theta}

\newcommand{\ph}{\varphi}

\title[Questions about Extreme Points]
{Questions about Extreme Points}
\author[K. M. Dyakonov]{Konstantin M. Dyakonov}

\address
{Departament de Matem\`atiques i Inform\`atica\\ 
Universitat de Barcelona, IMUB and BGSMath\\
Gran Via de les Corts Catalanes, 585\\ 
E-08007 Barcelona\\ 
Spain}

\address{\,\,\\
and}

\address{\quad\\
Instituci\'o Catalana de Recerca i Estudis Avan\c{c}ats (ICREA)\\ 
Pg. Llu\'is Companys, 23\\ 
E-08010 Barcelona\\ 
Spain}

\email{konstantin.dyakonov@icrea.cat}
\thanks{Supported in part by grants MTM2017-83499-P and PID2021-123405NB-I00 from El Ministerio de Ciencia e Innovaci\'on (Spain).}

\begin{document}
\begin{abstract}
We discuss the geometry of the unit ball---specifically, the structure of its extreme
points (if any)---in subspaces of $L^1$ and $L^\infty$ on the circle that are formed by functions with prescribed spectral gaps. A similar issue is considered for kernels of Toeplitz operators in $H^\infty$.
\end{abstract}

\maketitle

\section{Introduction}

Given a Banach space $X=(X,\|\cdot\|)$, we write  
$$\text{\rm ball}(X):=\{x\in X:\,\|x\|\le1\}.$$
An element $x$ of $\text{\rm ball}(X)$ is said to be an {\it extreme point} thereof if it is not expressible as $x=\f12(u+v)$ with two distinct points $u,v\in\text{\rm ball}(X)$. Clearly, every extreme point $x$ of $\text{\rm ball}(X)$ satisfies $\|x\|=1$. 

\par In what follows, the role of $X$ is played by certain function spaces on the circle $\T:=\{z\in\C:|z|=1\}$ which are defined in spectral terms. First of all, letting $m$ stand for the normalized arc length measure on $\T$, we introduce the (Lebesgue) spaces $L^p=L^p(\T,m)$ in the usual way, and we denote the standard $L^p$ norm by $\|\cdot\|_p$. Further, we recall that the {\it Fourier coefficients} of a function $f\in L^1$ are given by
$$\widehat f(k):=\int_\T\ov\ze^kf(\ze)\,dm(\ze),\qquad k\in\Z,$$
and the set 
$$\text{\rm spec}\,f:=\{k\in\Z:\,\widehat f(k)\ne0\}$$
is called the {\it spectrum} of $f$. 

\par For $1\le p\le\infty$, the {\it Hardy space} $H^p$ is then defined by
$$H^p:=\{f\in L^p:\,\text{\rm spec}\,f\subset\Z_+\},$$
where $\Z_+:=\{0,1,2,\dots\}$. (We also introduce the notation $\Z_-:=\Z\setminus\Z_+$ for future reference.) As usual, we may view elements of $H^p$ as holomorphic functions on the open unit disk when convenient; see \cite[Chapter II]{G} for the underlying theory and basic properties of $H^p$ spaces. 

\par More generally, given a nonempty set $\La\subset\Z_+$, we consider the {\it lacunary} (or {\it punctured}) {\it Hardy spaces}
$$H^p(\La):=\{f\in L^p:\,\text{\rm spec}\,f\subset\La\},\qquad1\le p\le\infty,$$
normed by $\|\cdot\|_p$ as before. We are concerned with the extreme points of $\text{\rm ball}(H^p(\La))$, so only the endpoint exponents $p=1$ and $p=\infty$ are of interest. Indeed, for $1<p<\infty$, the uniform convexity of $L^p$ implies that every unit-norm function is extreme.

\par In the classical setting, it is well known that the extreme points of $\text{\rm ball}(H^1)$ are precisely the outer functions $\mathcal F\in H^1$ with $\|\mathcal F\|_1=1$, whereas the extreme points of $\text{\rm ball}(H^\infty)$ are the functions $f\in H^\infty$ satisfying $\|f\|_\infty=1$ and 
\begin{equation}\label{eqn:logintdiv}
\int_\T\log(1-|f|)\,dm=-\infty.
\end{equation}
Both results can be found in \cite{dLR}; alternatively, see \cite[Chapter IV]{G} or \cite[Chapter 9]{H}. 

\par Recently, the author was able to establish the corresponding extreme point criteria in $H^p(\La)$, with $p=1,\infty$, under the hypothesis that the underlying set $\La$ is either small or large. Precisely speaking, it was assumed that either 
\begin{equation}\label{eqn:lasmall}
\#\La<\infty
\end{equation}
or 
\begin{equation}\label{eqn:lalarge}
\#(\Z_+\setminus\La)<\infty.
\end{equation}
In the case of $H^1(\La)$, the extreme points of the unit ball were described in \cite{DLac} under condition \eqref{eqn:lasmall}, and in \cite{DCR, DNear} under condition \eqref{eqn:lalarge}; the case of $H^\infty(\La)$ was treated in \cite{DAA} for both types of $\La$'s. 

\par Little seems to be known about the extreme points in $H^1(\La)$ and $H^\infty(\La)$ when neither \eqref{eqn:lasmall} nor \eqref{eqn:lalarge} holds. The questions we ask below are largely motivated by our curiosity in this regard. Sometimes, however, we find it natural to adopt a more general viewpoint. Namely, letting $\La$ be a subset of $\Z$ (not necessarily of $\Z_+$), we extend our attention to the {\it lacunary $L^p$ spaces}
$$L^p_\La:=\{f\in L^p:\,\text{\rm spec}\,f\subset\La\},\qquad p=1,\infty,$$
with norm $\|\cdot\|_p$. 

\section{Questions, problems, and a bit of discussion}

Here are some of the questions that puzzle us. 

\medskip\noindent\textbf{Question 1.} Given a set $\La\subset\Z$, which unit-norm functions from $L^1_\La$ (if any) are extreme points for $\text{\rm ball}(L^1_\La)$? Also, what are the extreme points of $\text{\rm ball}(L^\infty_\La)$?

\medskip Clearly, of concern are the cases that do not reduce to the existing results on $H^p(\La)$ as described above. Furthermore, it may well happen for a suitable $\La$ that $\text{\rm ball}(L^1_\La)$ has no extreme points at all. (A classical example is provided by taking $\La=\Z$, in which case $L^1_\La$ becomes the \lq\lq full" $L^1$.) In fact, the mere existence of extreme points seems to present a nontrivial problem, which we now state and discuss in some detail.

\medskip\noindent\textbf{Question 2.} For which sets $\La\subset\Z$ does $\text{\rm ball}(L^1_\La)$ possess an extreme point? In particular, for which sets $\La$ of the form 
\begin{equation}\label{eqn:lame}
\La=E\cup\Z_+,\quad\text{\rm with }E\subset\Z_-,
\end{equation}
does this happen?

\medskip Our interest in this last class of sets reflects an attempt to interpolate, so to speak, between $H^1$ and $L^1$ (i.e., between the cases $E=\emptyset$ and $E=\Z_-$), where two different things occur. Namely, the unit ball has plenty of extreme points in the former case, and none at all in the latter.

\par Now, let us say that a set $\La\subset\Z$ is {\it periodic} if there is a positive integer $n$ such that 
\begin{equation}\label{eqn:perset}
\La+n=\La
\end{equation}
(as usual, $\La+n$ stands for $\{k+n:k\in\La\}$). For instance, any arithmetic progression in $\Z$ is obviously periodic. 

\par To introduce another type of sets that we need here, we first recall the notation $M(\T)$ for the space of all finite Borel complex measures on $\T$. Also, for $\mu\in M(\T)$, we let $\text{\rm spec}\,\mu$ denote the set of those indices $k\in\Z$ for which $\widehat\mu(k):=\int_\T\ov z^kd\mu$ is nonzero. Finally, a subset $\La$ of $\Z$ is said to be a {\it Riesz set}, written as $\La\in\mathcal R$, if every measure $\mu\in M(\T)$ with $\text{\rm spec}\,\mu\subset\La$ is absolutely continuous with respect to $m$. 

\par The classical F. and M. Riesz theorem (see, e.g., \cite[Chapter II]{G}) tells us that $\Z_+\in\mathcal R$. A deeper study and further examples of Riesz sets can be found in \cite[Part One, Chapter 1]{HJ}. Among these examples are the $\La$'s given by \eqref{eqn:lame}, where $E$ is one of the following sets: 
$$\{-2^k:k\in\N\},\qquad\{-k^2:k\in\N\},\qquad\{-p:\,p\,\,\text{\rm prime}\}.$$
The next result provides a bit of information on Question 2 (but is a far cry from answering it completely).

\begin{thm} Let $\La\subset\Z$. If either $\La$ is periodic or $\#(\Z\setminus\La)<\infty$, then $\text{\rm ball}(L^1_\La)$ has no extreme points. On the other hand, if $\La\in\mathcal R$ then $\text{\rm ball}(L^1_\La)$ does possess extreme points.
\end{thm}

\begin{proof} Suppose that $\La$ is periodic, so that \eqref{eqn:perset} holds for some $n\in\N$. Now let $f\in L^1_\La$ be an arbitrary function with $\|f\|_1=1$. To show that $f$ is not an extreme point of $\text{\rm ball}(L^1_\La)$, it suffices to find a real-valued function $h\in L^\infty$ such that $fh\in L^1_\La$ and $h$ is nonconstant on the set $\{\ze\in\T:f(\ze)\ne0\}$. (The existence of such an $h$ is actually equivalent to the statement that $f$ is nonextreme for $\text{\rm ball}(L^1_\La)$. We refer to \cite[Chapter V, Section 9]{Gam} or \cite[Lemma 2.1]{DNear}, where the equivalence is proved in the context of $H^1$ and its subspaces; the case of a general subspace in $L^1$ is similar.) One possible choice is 
$$h(z)=\text{\rm Re}(z^n)=\f12\left(z^n+\ov z^n\right),\qquad z\in\T.$$
Indeed, the assumption that $\text{\rm spec}\,f\subset\La$ implies, in conjunction with \eqref{eqn:perset}, that 
$$\text{\rm spec}\,(z^nf)\subset\La\quad\text{\rm and}\quad
\text{\rm spec}\,(\ov z^nf)\subset\La.$$
Hence $\text{\rm spec}\,(fh)\subset\La$, so that $fh\in L^1_\La$. 

\par Now suppose that $\Z\setminus\La$ is a finite set, say, of cardinality $N$. Thus, 
\begin{equation}\label{eqn:zminusla}
\Z\setminus\La=\{k_1,\dots,k_N\},
\end{equation}
where the $k_j$'s are pairwise distinct integers. Once again, given an arbitrary unit-norm function $f$ in $L^1_\La$, we prove that $f$ is a nonextreme point of $\text{\rm ball}(L^1_\La)$ by constructing a real-valued function $h\in L^\infty$ that satisfies $fh\in L^1_\La$ and is nonconstant on the support of $f$. In fact, we claim that for a suitable nonzero vector 
\begin{equation}\label{eqn:vectal}
\al=(\al_1,\dots,\al_{2N+1})\in\R^{2N+1},
\end{equation}
the function 
$$h_\al(z):=\text{\rm Re}\left(\sum_{j=1}^{2N+1}\al_jz^j\right),\qquad z\in\T,$$
does the job. To check this, we associate with each vector \eqref{eqn:vectal} the numbers
\begin{equation}\label{eqn:galal}
\ga_\nu(\al):=\widehat{\left(fh_\al\right)}(k_\nu),\qquad\nu=1,\dots,N,
\end{equation}
and consider the linear map $S:\R^{2N+1}\to\R^{2N}$ defined by
$$S\al=
\left(\text{\rm Re}\,\ga_1(\al),\,\text{\rm Im}\,\ga_1(\al),\dots,
\text{\rm Re}\,\ga_N(\al),\,\text{\rm Im}\,\ga_N(\al)\right).$$
The rank of $S$ is of course bounded by $2N$, and we deduce from the rank-nullity theorem (see, e.g., \cite[p.\,63]{A}) that the kernel of $S$ has dimension at least $1$; in particular, the kernel is nontrivial. Now, if $\al\in\R^{2N+1}$ is a nonzero vector with $S\al=0$, then the numbers \eqref{eqn:galal} are all null, whence $fh_\al\in L^1_\La$. Also, the function $h_\al$ (which is obviously real-valued and bounded) is then nonconstant on any set $\mathcal E\subset\T$ with $m(\mathcal E)>0$. Our claim is thereby verified. 

\par Finally, suppose that $\La\in\mathcal R$. Consider the space $C:=C(\T)$ of all continuous functions on $\T$, and put 
$$C^\La:=\{f\in C:\,\text{\rm spec}\,f\subset\widetilde\La\},$$
where 
$$\widetilde\La:=\{-k:\,k\in\Z\setminus\La\}.$$
As usual, we identify the dual of $C$ with $M:=M(\T)$, the functional induced by a measure $\mu\in M$ being $g\mapsto\int_\T g\,d\mu$. The dual of the quotient space $C/C^\La$ is then $(C^\La)^\perp$, the annihilator of $C^\La$ in $M$. On the other hand, 
$$(C^\La)^\perp=\{\mu\in M:\,\text{\rm spec}\,\mu\subset\La\}.$$
This last set of measures embeds in $L^1$ (the $\mu$'s involved are absolutely continuous with respect to $m$ because $\La\in\mathcal R$), so it coincides with $L^1_\La$. Consequently, we have 
$$(C/C^\La)^*=(C^\La)^\perp=L^1_\La.$$
The existence of extreme points in $\text{\rm ball}(L^1_\La)$ is now guaranteed by the Krein--Milman theorem; see, e.g., \cite[Chapter 9]{H}. \end{proof}

\par Our next question is motivated by the conjecture---or perhaps a vague feeling---that if $\La\subset\Z_+$ and if $\La$ contains \lq\lq most" of $\Z_+$, then the extreme points of $\text{\rm ball}(H^1(\La))$ are \lq\lq not too far" from being outer functions. Indeed, when $\La$ is {\it all} of $\Z_+$, our space is just $H^1$ and its extreme points are precisely the outer functions of norm 1; see \cite{dLR}. Furthermore, it was shown in \cite{DNear} (see also \cite{DCR}) that if $\Z_+\setminus\La$ is a finite set, say with $\#(\Z_+\setminus\La)=N$, and if $f$ is an extreme point of $\text{\rm ball}(H^1(\La))$, then the inner factor of $f$ is necessarily a finite Blaschke product with at most $N$ zeros. In light of these facts, it seems tempting to conjecture that when $\Z_+\setminus\La$ is appropriately \lq\lq thin" (or \lq\lq sparse") in $\Z_+$, the inner factors corresponding to the extreme points of $\text{\rm ball}(H^1(\La))$ are still fairly \lq\lq tame," in some sense or other. It would be nice to have a rigorous result to that effect. 

\medskip\noindent\textbf{Question 3.} Suppose that $F$ is a suitably sparse (infinite) subset of $\Z_+$, and let $\La=\Z_+\setminus F$. What can we say about the inner factors of functions that arise as extreme points of $\text{\rm ball}(H^1(\La))$? To be more specific, what happens when $F$ is $\{2^k:\,k\in\Z_+\}$ or $\{2^{2^k}:\,k\in\Z_+\}$?

\medskip On the other hand, the case of $H^1(\La)$ where $\La$ (rather than $\Z_+\setminus\La$) is a sparse---say, Hadamard lacunary---subset of $\Z_+$ is also worth studying; that would provide a natural extension to what was done in \cite{DLac}. 

\par Turning to the $L^\infty$ part of Question 1, we now make a few observations pertaining to that setting. First we show that if $\La$ is obtained from $\Z$ by removing a finite number of elements, then the extreme points in $L^\infty_\La$ are precisely the unimodular functions, just as it happens for $L^\infty(=L^\infty_\Z)$. 

\begin{prop} Suppose that $\La\subset\Z$ and $\#(\Z\setminus\La)<\infty$. In order that a function $f\in L^\infty_\La$ with $\|f\|_\infty=1$ be an extreme point of $\text{\rm ball}(L^\infty_\La)$, it is necessary and sufficient that $|f|=1$ a.e. on $\T$. 
\end{prop}

\begin{proof} The sufficiency is obvious, since $L^\infty_\La\subset L^\infty$ and every unimodular function is an extreme point of $\text{\rm ball}(L^\infty)$. 
\par To prove the necessity, let \eqref{eqn:zminusla} be an enumeration of $\Z\setminus\La$. Now suppose $f$ is a unit-norm function in $L^\infty_\La$ that satisfies $|f|<1$ on a set of positive measure on $\T$. We then define $g:=1-|f|$, so that $g$ is a non-null function in $L^\infty$; clearly, we also have $g\ge0$ a.e. on $\T$. Further, with each vector 
$$\be=(\be_0,\be_1,\dots,\be_N)\in\C^{N+1}$$
we associate the polynomial 
$$p_\be(z):=\sum_{j=0}^N\be_jz^j,\qquad z\in\T,$$
and consider the linear map $T:\C^{N+1}\to\C^N$ that acts by the rule 
$$T\be=\left(\widehat{(gp_\be)}(k_1),\dots,\widehat{(gp_\be)}(k_N)\right).$$
The rank of $T$ being obviously bounded by $N$, we invoke the rank-nullity theorem to conclude that the kernel of $T$ is nontrivial. 
\par Now, if $\be\in\C^{N+1}$ is a nonzero vector with $T\be=0$, then the corresponding polynomial $p=p_\be$ is non-null and satisfies $gp\in L^\infty_\La\setminus\{0\}$. We may assume in addition that $\|p\|_\infty=1$, which yields
$$|f\pm gp|\le|f|+g|p|\le|f|+g=1$$
almost everywhere on $\T$. Consequently, $f+gp$ and $f-gp$ are two distinct points of $\text{\rm ball}(L^\infty_\La)$, and the identity 
$$f=\f12(f+gp)+\f12(f-gp)$$
shows that $f$ fails to be extreme for the ball.
\end{proof}

\par At the same time, it is not hard to produce a set $\La\subset\Z$ with $\sup\La=\infty$ and $\inf\La=-\infty$ for which $\text{\rm ball}(L^\infty_\La)$ has a much richer supply of extreme points. To this end, we first introduce a bit of terminology. Following \cite{HJ}, we say that a set $\La(\subset\Z)$ is a {\it $\mathcal D$-set} if it has the following property: whenever $\mu\in M(\T)$ is a measure with $\text{\rm spec}\,\mu\subset\La$ whose total variation $|\mu|$ assigns zero mass to a set of positive $m$-measure (length) on $\T$, we have $\mu=0$. 

\par As a classical example of a $\mathcal D$-set, we mention $\Z_+$; indeed, an $H^1$ function that vanishes on a set $\mathcal E\subset\T$ with $m(\mathcal E)>0$ must be null. For more sophisticated examples, we refer the reader to \cite[Part One, Chapter 1]{HJ}. In particular, it is shown there that if $E=\{-n^k:\,k\in\N\}$ with an integer $n\ge2$, then $E\cup\Z_+$ is a $\mathcal D$-set. 

\begin{prop} Let $\La$ be a $\mathcal D$-set. Suppose further that $f\in L^\infty_\La$ is a function with $\|f\|_\infty=1$ for which
\begin{equation}\label{eqn:alexpo}
m\left(\{\ze\in\T:\,|f(\ze)|=1\}\right)>0.
\end{equation}
Then $f$ is an extreme point of $\text{\rm ball}(L^\infty_\La)$. 
\end{prop}

\begin{proof} We want to check that the only function $g\in L^\infty_\La$ satisfying
\begin{equation}\label{eqn:fpmgleone}
\|f+g\|_\infty\le1\quad\text{\rm and}\quad\|f-g\|_\infty\le1
\end{equation}
is $g\equiv0$. Since 
$$|f|^2+|g|^2=\f12\left(|f+g|^2+|f-g|^2\right),$$
it follows from \eqref{eqn:fpmgleone} that $|g|^2\le1-|f|^2$ a.e. on $\T$. Consequently, $g=0$ a.e. on 
$$\mathcal E_f:=\{\ze\in\T:\,|f(\ze)|=1\},$$
while \eqref{eqn:alexpo} tells us that $m(\mathcal E_f)>0$. The desired conclusion that $g\equiv0$ is now ensured by the hypothesis that $\La$ is a $\mathcal D$-set. (To see why, identify $g$ with the measure $\mu_g\in M(\T)$ given by $d\mu_g=g\,dm$. Note also that 
$$\text{\rm spec}\,\mu_g=\text{\rm spec}\,g\subset\La$$
and use the identity $|\mu_g|(\mathcal E_f)=\int_{\mathcal E_f}|g|\,dm=0$ to deduce that $\mu_g$, and hence $g$, is null.) We are done. 
\end{proof}

\par We mention in passing that, by a theorem of Amar and Lederer (see \cite{AL}), the unit-norm $H^\infty$ functions that obey \eqref{eqn:alexpo} are precisely the {\it exposed points} of $\text{\rm ball}(H^\infty)$. (Recall that, for a Banach space $X$, a point $x$ in $\text{\rm ball}(X)$ is said to be {\it exposed} for the ball if there exists a functional $\phi\in X^*$ of norm $1$ such that the set $\{y\in\text{\rm ball}(X):\phi(y)=1\}$ equals $\{x\}$. It is well known, and easily shown, that every exposed point is extreme.) The following question might be of interest in this connection. 

\medskip\noindent\textbf{Question 4.} Does there exist a set $\La\subset\Z$ such that the extreme points of $\text{\rm ball}(L^\infty_\La)$ are characterized, among the unit-norm functions $f\in L^\infty_\La$, by condition \eqref{eqn:alexpo}?

\medskip From \eqref{eqn:alexpo}, we now turn to the weaker condition \eqref{eqn:logintdiv} which characterizes the extreme points of $\text{\rm ball}(H^\infty)$. This time, we ask whether the criterion remains unchanged for suitably perturbed $H^\infty$-spaces of the form $L^\infty_\La$, provided that $\La$ is \lq\lq not too different" from $\Z_+$. 

\medskip\noindent\textbf{Question 5.} For which sets $F\subset\Z_+$ is it true that \eqref{eqn:logintdiv} characterizes the extreme points $f$ of $\text{\rm ball}(H^\infty(\Z_+\setminus F))$? Also, for which sets $E\subset\Z_-$ does \eqref{eqn:logintdiv} characterize the extreme points $f$ of $\text{\rm ball}(L^\infty_\La)$, where $\La=E\cup\Z_+$?

\medskip The function $f$ to be tested is, of course, always assumed to be a unit-norm element of the space in question. Now, if $\#F<\infty$, then the corresponding extreme point criterion is indeed given by \eqref{eqn:logintdiv} (see \cite[Theorem 2.1]{DAA}), and a similar fact is true if $\#E<\infty$. The same criterion should apply when $F$ (resp., $E$) is appropriately sparse in $\Z_+$ (resp., $\Z_-$), and we would like to see a reasonably sharp sparseness condition that ensures this. 

\par We note, however, that taking $F$ to be the set of odd positive integers, we get $\Z_+\setminus F=2\Z_+$ and the extreme points $f$ of $\text{\rm ball}(H^\infty(2\Z_+))$ are again described by \eqref{eqn:logintdiv} (see \cite{DAA} for a more detailed discussion of this example). Thus, $F$ need not be any thinner than $\Z_+\setminus F$ in this situation. 

\par Going back to our description of the extreme points of $\text{\rm ball}(H^1(\La))$ and $\text{\rm ball}(H^\infty(\La))$, as obtained previously in the cases \eqref{eqn:lasmall} and \eqref{eqn:lalarge}, we now want to extend these results in yet another direction. 

\medskip\noindent\textbf{Question 6.} What happens to the results just mentioned, as well as to their $L^p_\La$ versions, in higher dimensions (say, on $\T^d$ in place of $\T$)? Also, what happens when passing from $\T$ to $\R$ (or $\R^d$)? 

\medskip Of course, the lacunary Hardy spaces $H^p(\La)$ (resp., the $L^p_\La$ spaces) on the torus $\T^d$ should be defined appropriately in terms of a given set of multi-indices $\La\subset\Z^d_+$ (resp., $\La\subset\Z^d$). In particular, the analogue of \eqref{eqn:lalarge} should now read $\#(\Z^d_+\setminus\La)<\infty$. 

\par Moving to the real line, we fix a closed set $\La\subset\R$ and define $L^p_\La=L^p_\La(\R)$ with $p=1,\infty$ as the space of all functions $f\in L^p(\R)$ whose Fourier transform $\widehat f$ vanishes on $\R\setminus\La$ (when $p=\infty$, we interpret $\widehat f$ in the sense of distributions). The lacunary Hardy spaces $H^p(\La)$ arise when $\La\subset[0,\infty)$. Now, as a natural counterpart of \eqref{eqn:lasmall}, we may impose the condition that $\La$ be a compact set of positive length; the corresponding Paley--Wiener type spaces $L^p_\La$ are actually of special interest. In the simplest case where $\La$ is an interval, the extreme (and exposed) points of $\text{\rm ball}(L^1_\La(\R))$ were characterized in \cite{DMRL2000}. A similar study of the \lq\lq second simplest" case, where $\La$ is made up of two disjoint intervals, was recently carried out in \cite{UZ} (also in the $L^1$ setting), and little---if anything---is known beyond that. 

\par Our last question deals with a different type of subspaces in $H^\infty$ (we are back to $\T$ now), where the structure of extreme points seems to be unclear. Given a function $\ph$ in $L^\infty=L^\infty(\T)$, we put
$$K_p(\ph):=\{f\in H^p:\,\ov{z\ph f}\in H^p\},\qquad1\le p\le\infty,$$
so that $K_p(\ph)$ is the kernel in $H^p$ of the Toeplitz operator with symbol $\ph$. 

\medskip\noindent\textbf{Question 7.} Let $\ph\in L^\infty$ and assume that $K_\infty(\ph)\ne\{0\}$. What are the extreme points of $\text{\rm ball}(K_\infty(\ph))$?

\medskip When $\ph=\ov\th$ for an inner function $\th$, $K_\infty(\ph)$ becomes the {\it model subspace} $H^\infty\cap\th\ov z\ov{H^\infty}$, and the problem of determining its extreme points was posed earlier in \cite{DIEOT}. Furthermore, if $\ph(z)=\ov z^{N+1}$ for some $N\in\Z_+$, then $K_\infty(\ph)$ coincides with $H^\infty(\La_N)$, where $\La_N:=\{0,1,\dots,N\}$, and is formed by the polynomials of degree at most $N$. In this last case, the extreme points are known (see \cite{DMRL2003} or \cite{DAA}). On the other hand, the extreme points of $\text{\rm ball}(K_1(\ph))$ admit a neat description for a general $\ph\in L^\infty$; this can be found in \cite{DPAMS}.

\par We remark, in conclusion, that there are related geometric concepts---such as exposed, strongly exposed, and {\it strong extreme} points of the unit ball---which are also worth studying in the context of lacunary $H^p$ or $L^p$ spaces, as well as in $K_p(\ph)$, with $p=1,\infty$. In fact, even for the usual (nonlacunary) $H^1$, the structure of its exposed points is far from being understood; the case of $H^1(\La)$ is touched upon in \cite{DLac, DNear} for the sets $\La$ that obey \eqref{eqn:lasmall} or \eqref{eqn:lalarge}. As regards strong extreme points, we refer to \cite{CT} for the definition and a characterization of these in the classical $H^p$ setting.

\smallskip{\it Acknowledgement.} I thank Aleksei Aleksandrov for calling my attention to periodic sets in connection with Question 2.

\section*{Competing Interests}

The author has no competing interests to declare that are relevant to the content of this article.

\end{document}